\documentclass[11pt,a4paper]{article}
\pdfoutput=1

\usepackage[intlimits]{amsmath}
\usepackage{amsfonts}
\usepackage{amssymb}
\usepackage{mathrsfs}
\usepackage{amsthm}
\usepackage[english]{babel}
\usepackage[latin1]{inputenc}
\usepackage[T1]{fontenc}
\usepackage{graphicx}
\usepackage{epstopdf}
\usepackage{authblk}

\def\a{\alpha}

\def\s{\sigma}

\def\ps{\psi}

\newcommand{\ti}[1]{\tilde{#1}}

\newcommand{\abs}[1]{\lvert#1\rvert}

\newcommand{\norm}[1]{\lVert#1\rVert}
\newcommand{\bnorm}[1]{\big\lVert#1\big\rVert}

\newcommand{\blangle}{\big{\langle}}
\newcommand{\Blangle}{\Big{\langle}}

\newcommand{\brangle}{\big{\rangle}}
\newcommand{\Brangle}{\Big{\rangle}}

\newcommand{\noi}{\noindent}
\newcommand{\vect}[1]{\mathbf{#1}}
\newcommand{\vectx}{\vect{x}}
\newcommand{\vecty}{\vect{y}}
\newcommand{\vectz}{\vect{z}}
\newcommand{\vectk}{\vect{k}}

\newcommand{\vectq}{\vect{q}}

\newcommand{\mcal}[1]{\mathcal{#1}}
\newcommand{\mscr}[1]{\mathscr{#1}}
\newcommand{\mscrF}{\mathscr{F}}
\newcommand{\mscrS}{\mathscr{S}}
\newcommand{\mbbN}{\mathbb{N}}
\newcommand{\mbbR}{\mathbb{R}}
\newcommand{\mbbC}{\mathbb{C}}

\newtheorem{definition}{Def\mbox{}inition}
\newtheorem{lemma}{Lemma}
\newtheorem{proposition}{Proposition}
\newtheorem{theorem}{Theorem}
\newtheorem{corollary}{Corollary}

\allowdisplaybreaks[4]

\begin{document}

\title{A novel def\-inition of real Fourier transform}

\author{\Large{Fulvio Sbis\`a\footnote{fulviosbisa@gmail.com ; https://orcid.org/0000-0002-6341-1785 .}}}

\affil{\normalsize{Departamento de F\'isica Te\'orica, Universidade do Estado do Rio de Janeiro,\\
Rua S\~{a}o Francisco Xavier 524, Maracan\~{a}, \\
CEP 20550-900, Rio de Janeiro -- RJ -- Brazil\thanks{On leave.}}
\and 
\normalsize{Faculdade Est\'{a}cio de S\'{a} -- Polo Jardim Camburi \\
Av.\ Dr.\ Herwan Modenese Wanderley, 1001 - Jardim Camburi, \\
CEP 29060-510, Vit\'{o}ria -- ES -- Brazil}}

\date{}

\vspace{.5cm}

\maketitle

\thispagestyle{empty}

\begin{abstract}
We propose a novel def\-inition of Fourier transform, with the property that the transform of a real function is again a real function (without doubling the number of real components). We prove the inversion theorem for the novel def\-inition, and show that it shares the good properties of the usual def\-inition.
\end{abstract}

\noi \textit{Keywords}: Fourier analysis; Fourier transform; functional analysis; modal expansion.

\section{Introduction}

Fourier analysis is undoubtedly one of the most used tools both in Physics, pure and applied Mathematics, and Engineering. It is indeed dif\-f\-icult to overemphasize its relevance. Its usefulness can be traced back to the invertibility of the Fourier transformation, and to the properties of its modes, the complex exponentials. On the one hand, the map $\a \to e^{i k \a}$ is a homomorphism of the additive group of the real numbers into the multiplicative group of the unit circle in the complex plane. On the other hand, the complex exponentials are eigenfunctions of the derivative operator. These properties for instance imply that, if a function can be Fourier-transformed and the convergence is good enough, a linear dif\-ferential operator with constant coef\-f\-icients can be mapped, in the reciprocal space, to an operator which multiplies the Fourier transform by a polynomial. The only strong limitation on the use of the Fourier transform and anti-transform is that they presuppose quite stringent decay properties.

When a function $f : \mbbR^{n} \to \mbbC$ is absolutely integrable, its Fourier transform $\mscr{F}[f]$ and anti-transform $\mscr{F}_{\!a}[f]$ are def\-ined via the integral representations
\begin{equation} \label{complex Fourier transform}
\mscr{F}[f](\vecty) = \frac{1}{\sqrt{(2 \pi)^{n}}} \int_{\mbbR^{n}} \! f(\vectx) \, e^{- i \vecty \cdot \vectx} \, d\vectx \quad ,
\end{equation}
and
\begin{equation} \label{complex Fourier antitransform}
\mscr{F}_{\!a}[f](\vectx) = \frac{1}{\sqrt{(2 \pi)^{n}}} \int_{\mbbR^{n}} \! f(\vecty) \, e^{i \vecty \cdot \vectx} \, d\vecty \quad .
\end{equation}
A few words are in order about notations and conventions. Throughout the article, bold-face letters denote vectors of $\mbbR^{n}$, and the central dot in $\vecty \cdot \vectx$ denotes the standard Euclidean inner product in $\mbbR^{n}$. The integrations are to be understood as Lebesgue integrals. The dependence of a function on real/complex numbers or vectors is indicated with round brackets, while the dependence on a function is indicated with square brackets. The symbols $\Re$ and $\Im$ indicate respectively the real and imaginary parts of a complex number. We use the ``symmetric'' choice for the normalization of the Fourier transform and anti-transform. Note that we adhere to the (admittedly confusing) practice of using the term ``transform'' to denote both the transformation (operator) and the function which is obtained applying the transformation.

It is a classic result \cite{Rudin 2} that, if $f$ and $\mscr{F}[f]$ are absolutely integrable, then $\big(\mscr{F}_{\!a} \circ \mscr{F}\big)[f] = f$ almost everywhere. Moreover, the linear operators $\mscr{F}$ and $\mscr{F}_{\!a}$ can be extended via linearity and continuity to unitary maps $L^{2}(\mbbR^{n}, \mbbC) \to L^{2}(\mbbR^{n}, \mbbC)$ (Plancherel theorem). The operators thus extended are one the inverse of the other, and one the adjoint of the other \cite{Reed Simon 2}. In another direction, $\mscr{F}$ and $\mscr{F}_{\!a}$ can be extended to bounded linear maps $L^{1}(\mbbR^{n}, \mbbC) \to C_{_{0}}(\mbbR^{n}, \mbbC)$ (Riemann-Lebesgue lemma), where $C_{_{0}}(\mbbR^{n}, \mbbC)$ denotes the space of continuous maps which decay to zero at inf\-inity. With these extensions in mind, it is more convenient (and arguably more elegant) to def\-ine the transform and anti-transform on the Schwartz space $\mscr{S}(\mbbR^{n}, \mbbC)$ of smooth functions of rapid decay.

\smallskip
The properties of the complex exponentials are in fact so convenient, that the def\-initions \eqref{complex Fourier transform}--\eqref{complex Fourier antitransform} are used also when dealing with real functions, by resorting to the well-known method of identifying a real function with a complex function whose imaginary part vanishes identically. Alternatively, it is of course possible to decompose the representation \eqref{complex Fourier transform} into its real and imaginary parts. In the latter approach, one def\-ines the Fourier transform as a \emph{couple} of real functions $\mcal{F}[f] = \big( \mcal{F}_{_{\!1}}[f] \, , \mcal{F}_{_{\!2}}[f] \big)\,$, where
\begin{subequations} \label{naive real Fourier transform}
\begin{align}
\mcal{F}_{_{\!1}}[f](\vecty) &= \frac{1}{\sqrt{(2 \pi)^{n}}} \int_{\mbbR^{n}} \! f(\vectx) \, \cos(\vecty \cdot \vectx) \, d\vectx \quad , \label{naive real Fourier transform 1} \\[2mm]
\mcal{F}_{_{\!2}}[f](\vecty) &= - \, \frac{1}{\sqrt{(2 \pi)^{n}}} \int_{\mbbR^{n}} \! f(\vectx) \, \sin(\vecty \cdot \vectx) \, d\vectx \quad , \label{naive real Fourier transform 2}
\end{align}
\end{subequations}
and def\-ines the anti-transform as follows
\begin{equation} \label{naive real Fourier antitransform}
\mcal{F}_{\!a} \big[ f_{_{1\!}}\, , f_{_{2}} \big](\vectx) = \frac{1}{\sqrt{(2 \pi)^{n}}} \int_{\mbbR^{n}} \! \Big( f_{_{1\!}}(\vecty) \, \cos(\vecty \cdot \vectx) - f_{_{2}}(\vecty) \, \sin(\vecty \cdot \vectx) \Big) \, d\vecty \quad .
\end{equation}
The minus signs in \eqref{naive real Fourier transform 2} and \eqref{naive real Fourier antitransform} are purely conventional, and stem from the desire of highlighting the link with the complex def\-initions. They can be harmlessly substituted by plus signs. Note that $\mcal{F}_{_{\!1}}[f]$ and $\mcal{F}_{_{\!2}}[f]$ are respectively even and odd with respect to a parity transformation $\vecty \to - \vecty$ (the same holds for the real and imaginary parts of $\mscr{F}[f]$ when $\Im[f] = 0$). So, for the inversion theorem to hold, it is necessary to restrict accordingly the co-domain of $\mcal{F}$ and the domain of $\mcal{F}_{\!a}$ (likewise, when $\Im[f] = 0\,$, it is necessary to restrict the co-domain of $\mscr{F}$ and the domain of $\mscr{F}_{\!a}$).

\smallskip
The def\-initions \eqref{complex Fourier transform}--\eqref{naive real Fourier antitransform} are widely used. No\-ne\-the\-less, as long as real functions are considered, some aspects of the formalism are not completely satisfying. It is somewhat annoying that the Fourier transform of a real function involves \emph{two} real functions. Also, it would be nice to avoid having to impose symmetry conditions on the components of the Fourier transform to guarantee the invertibility.

The aim of this article is to propose a simple novel def\-inition of Fourier transform of a real function (with associated inverse), which shares the good properties of the def\-initions \eqref{complex Fourier transform}--\eqref{naive real Fourier antitransform} while being free of the unsavory aspects mentioned above. The article is structured as follows: in section \ref{sec: inversion, involutivity and parity} we prove the inversion theorem and characterize the parity properties of the transformation; in section \ref{sec: unitarity, extensions and eigenfunctions} we prove the preservation of the $L^{2}$ inner product and consider the relevant domain and co-domain extensions; in section \ref{sec: convolution and Fourier transform} we discuss the interplay between the novel def\-inition and the notion of convolution. Furthermore, we include in Appendix \ref{app: modes} a discussion of the novel Fourier transformation seen as a modal expansion.

\section{Inversion theorem, involutivity and parity}
\label{sec: inversion, involutivity and parity}

We follow \cite{Reed Simon 2} and def\-ine the Fourier transform and anti-transform on the space of smooth functions of rapid decrease, with the understanding that the results will be subsequently extended domain- and co-domain-wise.

\subsection{Def\-inition and inversion theorem}

Let $f$ be a real function. Our aim is to provide a def\-inition of Fourier transform of $f$ which involves only \emph{one} real function, and moreover provide a def\-inition of Fourier anti-transform in such a way that the inversion theorem holds. The main condition we impose on the two transformations is that they be linear operators.

A simple way to obtain, in a linear fashion, a real function out of the couple $\mcal{F}_{_{\!1}}[f]$ and $\mcal{F}_{_{\!2}}[f]\,$, is to just sum them. The problem is that, in general, it is not possible to retrieve two addends after they have been summed. However, the circumstance that $\mcal{F}_{_{\!1}}[f]$ and $\mcal{F}_{_{\!2}}[f]$ have def\-inite parity makes it possible to retrieve them, furthermore in a linear fashion.

To avoid cumbersome expressions, let us indicate $\mscr{S}(\mbbR^{n}, \mbbR)$ simply with $\mscr{S}$, and call $\mscr{S}_{\!sa}$ the following ``symmetric-antisymmetric'' linear subspace of $\mscr{S} \oplus \mscr{S}$:
\begin{multline*}
\mscr{S}_{\!sa} = \big\{ (g_{_{1\!}}\, , g_{_{2}}) \in \mscr{S}(\mbbR^{n}, \mbbR) \oplus \mscr{S}(\mbbR^{n}, \mbbR) \mid \\
\mid g_{_{1\!}}(\vectx) = g_{_{1\!}}(-\vectx) \,\,\, , \,\, g_{_{2\!}}(\vectx) = - g_{_{2\!}}(-\vectx) \,\,\, , \,\, \forall \, \vectx \in \mbbR^{n} \Big\} \quad .
\end{multline*}
We note in passing that the Fourier transform $\mcal{F}$ and anti-transform $\mcal{F}_{a}$ are linear isomorphisms
\begin{equation*}
\mscr{S} \xrightarrow[]{ \,\,\, \mcal{F} \,\,\, } \mscr{S}_{\!sa} \xrightarrow[]{ \,\,\, \mcal{F}_{a} \,\,\, } \mscr{S} \quad .
\end{equation*}
Let us then def\-ine the ``sum'' and ``parity decomposition'' operators
\begin{align*}
S &: \mscr{S}_{\!sa} \to \mscr{S} \quad , & D &: \mscr{S} \to \mscr{S}_{\!sa} \quad ,
\end{align*}
via the relations
\begin{equation} \label{sum operator}
S \big[ \, g_{_{1\!}}\, , g_{_{2\!}} \, \big](\vectx) = g_{_{1\!}}(\vectx) + g_{_{2}}(\vectx) \quad ,
\end{equation}
and
\begin{equation} \label{decomposition operator}
D[ \, g \, ](\vectx) = \bigg( \, \frac{1}{2} \Big( g(\vectx) + g(-\vectx) \Big) \,\, , \,\, \frac{1}{2} \Big( g(\vectx) - g(-\vectx) \Big) \bigg) \quad .
\end{equation}
It is immediate to recognize that $S$ and $D$ are linear isomorphisms, and that they are one the inverse of the other. We then def\-ine
\begin{definition}[real Fourier transform and anti-transform] \label{def: real Fourier transform and anti-transf}
We call respectively \emph{real Fourier transform} and \emph{real Fourier anti-transform} the operators
\begin{equation*}
\mscr{F}_{_{\!\mbbR}} \,\, , \, \mscr{F}^{a}_{_{\!\mbbR}} \,\, : \,\, \mscr{S} \to \mscr{S} \quad ,
\end{equation*}
def\-ined by
\begin{align} \label{real Fourier transf and anti-transf}
\mscr{F}_{_{\!\mbbR}} &= S \circ \mcal{F} \quad , & \mscr{F}^{a}_{_{\!\mbbR}} &= \mcal{F}_{\!a} \circ D \quad .
\end{align}
\end{definition}
In words, the real Fourier transform $\mscr{F}_{_{\!\mbbR}}[f]$ is obtained by taking the Fourier transform $\mcal{F}$, and then summing the two components $\mcal{F}_{_{\!1}}[f]$ and $\mcal{F}_{_{\!2}}[f]\,$; the real Fourier anti-transform $\mscr{F}^{a}_{_{\!\mbbR}}[g]$ is obtained by f\-irst splitting $g$ into its parity-even and party-odd parts, and then applying the anti-transform $\mcal{F}_{\!a}\,$. It is not dif\-f\-icult to show that the inversion theorem holds:
\begin{theorem}[inversion theorem]
The linear operators $\mscr{F}_{_{\!\mbbR}}$ and $\mscr{F}^{a}_{_{\!\mbbR}}$ are isomorphisms, and
\begin{align*}
\mscr{F}_{_{\!\mbbR}}^{a} \circ \mscr{F}_{_{\!\mbbR}} &= \textup{Id} \quad , & \mscr{F}_{_{\!\mbbR}} \circ \mscr{F}^{a}_{_{\!\mbbR}} &= \textup{Id} \quad ,
\end{align*}
where $\textup{Id}$ is the identity operator on $\mscr{S}$.
\end{theorem}
\begin{proof}
$\mscr{F}_{_{\!\mbbR}}$ and $\mscr{F}^{a}_{_{\!\mbbR}}$ are compositions of isomorphisms, so they are isomorphisms themselves. Furthermore, since $\mcal{F}$ and $\mcal{F}_{a}$ are one the inverse of the other, we have
\begin{equation*}
\mscr{F}^{a}_{_{\!\mbbR}} \circ \mscr{F}_{_{\!\mbbR}} = \mcal{F}_{\!a} \circ \big( D \circ S \big) \circ \mcal{F} = \mcal{F}_{\!a} \circ \textup{Id}_{sa} \circ \mcal{F} = \textup{Id} \quad , 
\end{equation*}
and
\begin{equation*}
\mscr{F}_{_{\!\mbbR}} \circ \mscr{F}_{_{\!\mbbR}}^{a} = S \circ \big( \mcal{F} \circ \mcal{F}_{\!a} \big) \circ D = S \circ \textup{Id}_{sa} \circ D = \textup{Id} \quad , 
\end{equation*}
where $\textup{Id}_{sa}$ denotes the identity operator on $\mscr{S}_{\!sa}\,$.
\end{proof}

\subsection{Integral representation and involutivity}

From the abstract def\-inition \eqref{real Fourier transf and anti-transf}, we easily obtain the explicit integral representation
\begin{equation} \label{real Fourier transform}
\mscr{F}_{_{\!\mbbR}}[f](\vecty) = \frac{1}{\sqrt{(2 \pi)^{n}}} \int_{\mbbR^{n}} \! f(\vectx) \, \Big( \cos (\vecty \cdot \vectx) - \sin (\vecty \cdot \vectx) \Big) \, d\vectx \quad .
\end{equation}
The evaluation of the explicit representation of the real Fourier \emph{anti}-trans\-form reveals the following result:
\begin{proposition}[involutivity of the real Fourier transform] \label{prop: involutivity}
The real Fourier anti-transform $\mscr{F}_{_{\!\mbbR}}^{a}$ coincides with the real Fourier transform $\mscr{F}_{_{\!\mbbR}}\,$. In other words, $\mscr{F}_{_{\!\mbbR}}$ is an involution:
\begin{equation} \label{involutivity}
\mscr{F}_{_{\!\mbbR}} \circ \mscr{F}_{_{\!\mbbR}} = \textup{Id} \quad .
\end{equation}
\end{proposition}
\begin{proof}
Taking into account the expressions \eqref{naive real Fourier antitransform} and \eqref{decomposition operator}, the def\-inition \eqref{real Fourier transf and anti-transf} of $\mscr{F}_{_{\!\mbbR}}^{a}$ gives
\begin{multline*}
\mscr{F}_{_{\!\mbbR}}^{a}[f](\vectx) = \frac{1}{\sqrt{(2 \pi)^{n}}} \int_{\mbbR^{n}} \bigg[ \, \frac{1}{2} \, \Big( f(\vecty) + f(-\vecty) \Big) \, \cos(\vecty \cdot \vectx) \,\, - \\
- \frac{1}{2} \, \Big( f(\vecty) - f(-\vecty) \Big) \, \sin(\vecty \cdot \vectx) \, \bigg] \, d\vecty = \\
= \frac{1}{\sqrt{(2 \pi)^{n}}} \, \Bigg[ \, \frac{1}{2} \int_{\mbbR^{n}} f(\vecty) \Big( \cos(\vecty \cdot \vectx) - \sin(\vecty \cdot \vectx) \Big) \, d\vecty \,\, + \\
+ \frac{1}{2} \, \int_{\mbbR^{n}} f(-\vecty) \Big( \cos(\vecty \cdot \vectx) + \sin(\vecty \cdot \vectx) \Big) \, d\vecty \, \Bigg] \quad .
\end{multline*}
Performing in the second integral the change of integration variables $\vecty \to - \vecty\,$, we obtain the thesis.
\end{proof}

This result is somewhat unexpected, but very satisfactory. The asymmetry between Fourier transform and anti-transform completely disappears for real functions. Note that, if we adopted a non-symmetric choice of normalization for $\mcal{F}$ and $\mcal{F}_{a}\,$, we would obtain that the operator $\mscr{F}_{_{\!\mbbR}}^{a}$ is proportional to $\mscr{F}_{_{\!\mbbR}}\,$, but not equal. Such a situation seems quite artif\-icial, so the proposition \ref{prop: involutivity} provides, a posteriori, an additional motivation for the ``symmetric'' normalization choice.

Note that, as is true for the expressions \eqref{naive real Fourier transform 2}--\eqref{naive real Fourier antitransform}, also for the real Fourier transform $\mscr{F}_{_{\!\mbbR}}$ the minus sign in \eqref{real Fourier transform} is purely conventional, and can harmlessly turned into a plus sign.

\subsection{Fourier transform and parity}

Let us now investigate the relationship between the parity of a real function and that of its real Fourier transform. We say that a function $f$ is parity-even if for every $\vectx \in \mbbR^{n}$ we have $f(-\vectx) = f(\vectx)\,$; likewise, we say that a function $f$ is parity-odd if for every $\vectx \in \mbbR^{n}$ we have $f(-\vectx) = -f(\vectx)\,$.

\smallskip
The main result in this regard is:
\begin{proposition} \label{prop: parity properties}
Let $f \in \mscrS$. Then
\begin{enumerate}
\item $f$ is parity-even if and only if $\mscr{F}_{_{\!\mbbR}}[f]$ is parity-even;
\item $f$ is parity-odd if and only if $\mscr{F}_{_{\!\mbbR}}[f]$ is parity-odd.
\end{enumerate}
\end{proposition}
\begin{proof}
Let $f$ be parity-even. For every $\vecty \in \mbbR^{n}$, the function $f(\vectx) \, \sin (\vecty \cdot \vectx)$ is parity-odd as a function of $\vectx\,$, so its integral in $d\vectx$ over $\mbbR^{n}$ vanishes. The integral representation \eqref{real Fourier transform} then gives
\begin{equation*}
\mscr{F}_{_{\!\mbbR}}[f](\vecty) = \frac{1}{\sqrt{(2 \pi)^{n}}} \int_{\mbbR^{n}} \! f(\vectx) \, \cos (\vecty \cdot \vectx) \, d\vectx \quad ,
\end{equation*}
and the right hand side is a parity-even function of $\vecty\,$. If $f$ is parity-odd, analogously we have
\begin{equation*}
\mscr{F}_{_{\!\mbbR}}[f](\vecty) = - \, \frac{1}{\sqrt{(2 \pi)^{n}}} \int_{\mbbR^{n}} \! f(\vectx) \, \sin (\vecty \cdot \vectx) \, d\vectx \quad ,
\end{equation*}
and the right hand side is a parity-odd function of $\vecty\,$. The ``only if'' part of the thesis follows from the property \eqref{involutivity}.
\end{proof}

We conclude that, if a real function has a def\-inite parity, then its parity coincides with that of its real Fourier transform. Note that this property is not specif\-ic to the real Fourier transform, since it is true also for the Fourier transform $\mcal{F}$. However, in the latter case there is a dif\-ference between the even and odd cases, since a parity-even $f$ has $\mcal{F}_{_{\!2}}[f] = 0\,$, while a parity-odd $f$ has $\mcal{F}_{_{\!1}}[f] = 0\,$. For what concerns the complex Fourier transform, the situation is similar: if a real function $f$ is parity-even, then
\begin{align} \label{real vs complex Fourier transf parity even}
\mscr{F}_{_{\!\mbbR}}[f] &= \Re \big[ \mscr{F}[f] \big] \quad , & \Im \big[ \mscr{F}[f] \big] &= 0 \quad ;
\end{align}
if a real function $f$ is parity-odd, then
\begin{align} \label{real vs complex Fourier transf parity odd}
\mscr{F}_{_{\!\mbbR}}[f] &= \Im \big[ \mscr{F}[f] \big] \quad , & \Re \big[ \mscr{F}[f] \big] &= 0 \quad .
\end{align}

\section{Unitarity, extensions and eigenfunctions}
\label{sec: unitarity, extensions and eigenfunctions}

We now discuss some properties of the real Fourier transform as a map between normed spaces and inner-product spaces. To lighten the notation somewhat, henceforth we denote with $L^{2}$ and $L^{1}$ respectively the spaces $L^{2}(\mbbR^{n}, \mbbR)$ and $L^{1}(\mbbR^{n}, \mbbR)$ of square-integrable and absolutely-integrable real functions (quotiented by the equivalence relation which identif\-ies functions which coincide almost everywhere). We also denote with $C_{_{0\!}}$ the space $C_{_{0}}(\mbbR^{n}, \mbbR)$ of real continuous functions which decay to zero at inf\-inity.

We recall that $L^{2}$ is a real Hilbert space when equipped with the inner product
\begin{equation*}
\blangle f , g \brangle = \int_{\mbbR^{n}} f(\vectx) \, g(\vectx) \, d\vectx \quad ,
\end{equation*}
while $L^{1}$ and $C_{_{0\!}}$ are real Banach spaces when equipped respectively with the norms
\begin{align*}
\norm{f}_{_{1\!}} &= \int_{\mbbR^{n}} \abs{f(\vectx)} \, d\vectx \quad , & \norm{f}_{_{\infty\!}} &= \sup_{\vectx \in \mbbR^{n}} \, \abs{f(\vectx)} \quad .
\end{align*}
We indicate with $\norm{ \, }_{_{2\!}}$ the norm induced by the inner product $\blangle \, ,  \brangle\,$.

\subsection{Unitarity}

It is well known \cite{Reed Simon 2} that the complex Fourier transform $\mscr{F}$, seen as a map
\begin{equation*}
\big( \mscr{S}(\mbbR^{n}, \mbbC) \, , \langle \, , \rangle_{_{2}} \big) \to \big( \mscr{S}(\mbbR^{n}, \mbbC) \, , \langle \, , \rangle_{_{2}} \big) \quad ,
\end{equation*}
preserves the inner product
\begin{equation*}
\langle f , g \rangle_{_{2}} = \int_{\mbbR^{n}} f^{\ast}(\vectx) \, g(\vectx) \, d\vectx \quad .
\end{equation*}
In terms of the Fourier transform $\mcal{F}$, this means that, for every couple of real functions $f , g \in \mscr{S}$, we have
\begin{equation*}
\int_{\mbbR^{n}} \Big( \mcal{F}_{_{\!1}}[f](\vectx)\, \mcal{F}_{_{\!1}}[g](\vectx) + \mcal{F}_{_{\!2}}[f](\vectx)\, \mcal{F}_{_{\!2}}[g](\vectx) \Big) \, d\vectx = \int_{\mbbR^{n}} f(\vectx)\, g(\vectx)\, d\vectx \quad .
\end{equation*}
This relation can be written in the form
\begin{equation}
\Blangle \mcal{F}[f] \, , \, \mcal{F}[g] \Brangle_{_{\!\!\oplus}} = \blangle f , g \brangle \quad ,
\end{equation}
by introducing the real inner product $\blangle \, , \brangle_{_{\!\oplus}} : \mscr{S}_{\!sa} \times \mscr{S}_{\!sa} \to \mbbR$ as follows
\begin{equation*}
\Blangle \big( f_{_{1}} , f_{_{2}} \big) , \big( g_{_{1}} , g_{_{2}} \big) \Brangle_{_{\!\!\oplus}} = \int_{\mbbR^{n}} \Big( f_{_{\!1\!}}(\vectx) \, g_{_{1\!}}(\vectx) + f_{_{\!2}}(\vectx) \, g_{_{2}}(\vectx) \Big) \, d\vectx \quad .
\end{equation*}
So the Fourier transform $\mcal{F}$, seen as a map $\big( \mscr{S} \, , \langle \, , \rangle \big) \to \big( \mscr{S}_{\!sa} \, , \langle \, , \rangle_{_{\!\oplus}} \big)\,$, preserves the inner product.

\smallskip
The main result in this connection is that the real Fourier transform $\mscr{F}_{_{\!\mbbR}}$ preserves the inner product $\langle \, , \rangle \, $:
\begin{proposition}[unitarity of the real Fourier transform] \label{prop: unitarity}
The real Fourier transform, seen as a map between real inner product spaces
\begin{equation*}
\mscrF_{_{\!\mbbR}} : \big( \mscr{S} , \langle \, , \rangle \big) \to \big( \mscr{S} , \langle \, , \rangle \big) \quad ,
\end{equation*}
is an inner-product-preserving (i.e., unitary) map. That is, for every $f , g \in \mscr{S}$ we have
\begin{equation} \label{real Fourier transf inner preserving}
\blangle \, \mscr{F}_{_{\!\mbbR}}[f]\, , \, \mscr{F}_{_{\!\mbbR}}[g] \, \brangle = \blangle f , g \brangle \quad .
\end{equation}
In particular, for every $f \in \mscrS$ we have
\begin{equation} \label{real Fourier transf norm preserving}
\bnorm{\mscr{F}_{_{\!\mbbR}}[f]}_{_{2\!}} = \norm{f}_{_{2\!}} \quad ,
\end{equation}
so $\mscr{F}_{_{\!\mbbR}}$ is an isometry.
\end{proposition}
\begin{proof}
Recall that, by def\-inition, $\mscr{F}_{_{\!\mbbR}} = S \circ \mcal{F}$. We prove that $S$, seen as a map $\big( \mscr{S}_{\!sa} \, , \langle \, , \rangle_{_{\!\oplus}} \big) \to \big( \mscr{S} \, , \langle \, , \rangle \big)\,$, preserves the inner product. To this aim, let $\big( f_{_{\!1}} , f_{_{2}} \big) \, , \big( g_{_{1}} , g_{_{2}} \big) \in \mscr{S}_{\!sa}\,$. We have
\begin{multline*}
\Blangle S \big[ \big( f_{_{\!1}} , f_{_{2}} \big) \big] \, , S \big[ \big( g_{_{1}} , g_{_{2}} \big) \big] \Brangle = \int_{\mbbR^{n}} \big( f_{_{\! 1\!}}(\vectx) + f_{_{\!2}}(\vectx) \big) \, \big( g_{_{1\!}}(\vectx) + g_{_{2\!}}(\vectx) \big) \, d\vectx = \\
= \Blangle \big( f_{_{\!1}} , f_{_{2}} \big) \, , \, \big( g_{_{1}} , g_{_{2}} \big) \Brangle_{_{\!\oplus}} + \int_{\mbbR^{n}} \Big( f_{_{\! 1\!}}(\vectx) \, g_{_{2\!}}(\vectx) + f_{_{2\!}}(\vectx) \, g_{_{1\!}}(\vectx) \Big) \, d\vectx \quad ,
\end{multline*}
and the integral in the second line vanishes owing to the fact that the products $f_{_{\! 1}} \, g_{_{2\!}}$ and $f_{_{2}} \, g_{_{1\!}}$ are parity-odd functions. Since $\mscr{F}_{_{\!\mbbR}}$ is the composition of inner-product-preserving maps, it is inner-product-preserving itself.
\end{proof}

It follows in particular that $\mscrF_{_{\!\mbbR}}$ is symmetric as a map between real inner product spaces:
\begin{corollary} \label{cor: symmetric}
The real Fourier transform $\mscr{F}_{_{\!\mbbR}}$, seen as a map between real inner product spaces $\big( \mscrS , \langle \, , \rangle \big) \to \big( \mscrS , \langle \, , \rangle \big)\,$, is symmetric. That is, for every $f , g \in \mscr{S}$ we have
\begin{equation} \label{real Fourier transf symmetric}
\blangle f \, , \, \mscr{F}_{_{\!\mbbR}}[g]\, \brangle = \blangle \mscr{F}_{_{\!\mbbR}}[f] \, , \, g \brangle \quad .
\end{equation}
\end{corollary}
\begin{proof}
The thesis follows trivially from the fact that $\mscr{F}_{_{\!\mbbR}}$ preserves the inner product and is an involution.
\end{proof}

\subsection{Extension theorems}

As we mentioned above, the complex Fourier transform, def\-ined as a map $\mscr{F}: \mscr{S}(\mbbR^{n}, \mbbC) \to \mscr{S}(\mbbR^{n}, \mbbC)\,$, can be extended by enlarging the domain and co-domain. The two relevant extensions are that to a map $L^{2}(\mbbR^{n}, \mbbC) \to L^{2}(\mbbR^{n}, \mbbC)\,$, and that to a map $L^{1}(\mbbR^{n}, \mbbC) \to C_{_{0}}(\mbbR^{n}, \mbbC)\,$.

These extensions rely on the continuous linear extension theorem (also known as the ``B.L.T.\ theorem'', \cite{Reed Simon 1}). Given two Banach spaces $(V, \norm{ \, }_{a})$ and $(W, \norm{ \, }_{b})\,$, the theorem asserts that a bounded linear map $B \to W$, where $B$ is dense in $V$, can be uniquely extended to a bounded linear map $V \to W$, preserving the operator norm. For what concerns the complex Fourier transform, the crucial observation is that $\mscr{S}(\mbbR^{n}, \mbbC)$ is both dense in $\big( L^{2}(\mbbR^{n}, \mbbC)\, , \norm{ \, }_{_{2}} \big)$ and in $\big( L^{1}(\mbbR^{n}, \mbbC)\, , \norm{ \, }_{_{1}} \big)\,$, and is a subspace of $C_{_{0}}(\mbbR^{n}, \mbbC)\,$.

\smallskip
Two analogous results hold for the real Fourier transform:
\begin{theorem}[Plancherel, real case] \label{th: Plancherel, real case}
The real Fourier transform $\mscr{F}_{_{\!\mbbR\!}}$ extends uniquely to a unitary map $(L^{2} , \norm{ \, }_{_{2}}) \to (L^{2} , \norm{ \, }_{_{2}})\,$. The extension is an involution and is self-adjoint.
\end{theorem}

\begin{theorem}[Riemann-Lebesgue lemma, real case] \label{th: Riemann-Lebesgue, real case}
The real Fourier transform $\mscr{F}_{_{\!\mbbR\!}}$ extends uniquely to a bounded linear map $(L^{1} , \norm{ \, }_{_{1}}) \to (C_{_{0}} , \norm{ \, }_{_{\infty}})\,$.
\end{theorem}

The proofs are completely analogous to the ones for the complex case (see, e.g., \cite{Reed Simon 2}). Indeed, the continuous linear extension theorem works alike for complex and real Banach spaces. Furthermore, also in the real case $\mscr{S}$ is both dense in $\big( L^{2}, \norm{ \, }_{_{2}} \big)$ and in $\big( L^{1}, \norm{ \, }_{_{1}} \big)\,$, and is a subspace of $C_{_{0\!}}\,$. For what concerns the extension to $L^{2}$, the continuous linear extension theorem applies because the proposition \ref{prop: unitarity} implies that $\mscr{F}_{_{\!\mbbR}}\,$, seen as a map $(\mscr{S}, \norm{ \, }_{_{2}}) \to (\mscr{S}, \norm{ \, }_{_{2}})\,$, is a bounded linear map with unit norm. For what concerns the extension to $L^{1}$, it is easy to check that the real Fourier transform, seen as a map $(\mscr{S}, \norm{ \, }_{_{1}}) \to (\mscr{S}, \norm{ \, }_{_{\infty}})\,$, satisf\-ies the bound
\begin{equation*}
\bnorm{\mscr{F}_{_{\!\mbbR}}[f]}_{_{\infty}} \leq \frac{\sqrt{2}}{\sqrt{(2 \pi)^{n}}} \,\, \norm{f}_{_{1}} \quad ,
\end{equation*}
so it is a bounded linear map.

The only dif\-ference with the complex case is that, for what concerns the theorem \ref{th: Plancherel, real case}, the extended map is also an involution, and is self-adjoint (while in the complex case $\mscr{F}^{-1} = \mscr{F}_{a}$ and $\mscr{F}^{\dag} = \mscr{F}_{a}$). These properties follow from the facts that $\mscr{F}_{_{\!\mbbR}} : \mscr{S} \to \mscr{S}$ is an involution (proposition \ref{prop: involutivity}) and symmetric (corollary \ref{cor: symmetric}). On the other hand, as in the complex case, the extension $L^{1} \to C_{_{0}}$ in not surjective.
 
\subsection{Eigenfunctions}

Let us temporarily restrict our considerations to the case $n = 1\,$. It is known \cite{Pinsky 2002} that the Hermite functions
\begin{equation}
\ps_{k}(x) = \frac{(-1)^{k}}{\sqrt{2^{k}\, k! \, \sqrt{\pi} \, }} \,\, e^{\frac{x^{2}}{2}} \, \frac{d^{k}}{dx^{k}} \, e^{- x^{2}} \quad ,
\end{equation}
where the index $k$ runs over $\mbbN\,$, satisfy the relation
\begin{equation}
\frac{1}{\sqrt{2 \pi}} \int_{\mbbR} e^{-ixy} \, \ps_{k}(x)\, dx = (-i)^{k}\, \ps_{k}(y) \quad ,
\end{equation}
and therefore are eigenfunctions of the complex Fourier transform:
\begin{equation}
\mscr{F} \big[ \ps_{k} \big] = (-i)^{k}\, \ps_{k} \quad .
\end{equation}
There are four eigenvalues, $1\,$, $-i\,$, $-1$ and $i\,$, and the four eigenspaces in $\mscr{S}(\mbbR\, , \mbbR)$ are inf\mbox{}initely degenerate. The family of Hermite functions then gets naturally partitioned into four sub-families, according to their eigenvalue. In particular, for $m \in \mbbN$ we have
\begin{align}
\mscrF \big[ \ps_{2m} \big] &= (-1)^{m}\, \ps_{2m} \quad , & \mscrF \big[ \ps_{2m + 1} \big] &= i \, (-1)^{m + 1}\, \ps_{2m + 1} \quad ,
\end{align}
so the eigenvalues of the Hermite functions have a four-fold periodicity as the index of the latter runs over the natural numbers.

\smallskip
For what concerns the real Fourier transform $\mscr{F}_{_{\!\mbbR}}\,$, it is useful to recall that the Hermite functions have def\-inite parity, and in particular those of even index are parity-even, while those of odd index are parity-odd. The relations \eqref{real vs complex Fourier transf parity even} and \eqref{real vs complex Fourier transf parity odd} then imply that the Hermite functions are eigenfunctions of the real Fourier transform $\mscr{F}_{_{\!\mbbR}}$ as well, with eigenvalues equal to $+1$ and $-1\,$. Indeed, for $m \in \mbbN$ we have
\begin{align}
\mscr{F}_{_{\!\mbbR}} \big[ \ps_{2m} \big] &= (-1)^{m}\, \ps_{2m} \quad , & \mscr{F}_{_{\!\mbbR}} \big[ \ps_{2m + 1} \big] &= (-1)^{m + 1}\, \ps_{2m + 1} \quad .
\end{align}
The family of Hermite functions in this case gets partitioned into two sub-families, according to their eigenvalue, which are again inf\-initely degenerate. Nevertheless, the eigenvalues of the Hermite functions, as eigenfunctions of $\mscr{F}_{_{\!\mbbR}}\,$, still have a four-fold periodicity as their index runs over the natural numbers, with the sign pattern $+ - - \, +$ repeating indef\mbox{}initely.

\section{Convolution and Fourier transform}
\label{sec: convolution and Fourier transform}

\subsection{The notion of convolution}

We recall the notion of convolution:
\begin{definition}[convolution] \label{def: convolution}
Let $f\, , g \in \mscr{S}$. We call \emph{convolution of $g$ by $f$} the function $f \ast g \in \mscr{S}$ def\mbox{}ined as follows
\begin{equation} \label{convolution Schwartz space}
(f \ast g) (\vectx) = \int_{\mbbR^{n}} f(\vectx - \vecty) \, g(\vecty) \, d\vecty \quad .
\end{equation}
\end{definition}
\noi In the complex case where $f\, , g \in \mscr{S}(\mbbR^{n}, \mbbC)\,$, the def\-inition of convolution is formally the same, in other words is given by \eqref{convolution Schwartz space} without any complex conjugation appearing. Of course, in that case the product inside the integral is a product of complex numbers.

The convolution is a binary operation. It is well-known that, both in the real and in the complex case, this operation is commutative, associative and distributes with respect to the point-wise addition between functions \cite{Reed Simon 2}. In particular, the commutative property implies that we can speak simply of convolution of two functions, without specifying which one convolutes which. In this section, for additional clarity we indicate with a central dot the point-wise product between functions (both for the real and the complex case).

\smallskip
The operation of convolution displays a noteworthy interplay with the operation of Fourier transformation. For what concerns the complex Fourier transform, the main result in this sense is the following:
\begin{proposition} \label{prop: complex Fourier transform of product and convolution}
Let $f , g \in \mscr{S}(\mbbR^{n}, \mbbC)\,$. Then
\begin{align}
\frac{1}{\sqrt{(2 \pi)^{n}}} \,\, \mscrF \big[ f \ast g \big] &= \mscrF \big[ f \big] \cdot \mscrF \big[ g \big] \quad , \label{complex Fourier transform of convolution} \\[2mm]
\mscrF \big[ f \cdot g \big] &= \frac{1}{\sqrt{(2 \pi)^{n}}} \,\, \mscrF \big[ f \big] \ast \mscrF \big[ g \big] \quad . \label{complex Fourier transform of product}
\end{align}
\end{proposition}
\begin{proof}
See \cite{Reed Simon 2}.
\end{proof}
\noi In words, these relations mean that, apart from a numerical factor, the complex Fourier transform turns convolutions into products and products into convolutions. Similar properties, dual to these, hold for the complex Fourier anti-transform.

\subsection{Convolution and real Fourier transform}

The corresponding relations for the real Fourier transform assume a more complicated form. Before stating the proposition, let us introduce a dedicated notation and establish some preliminary results.

Given a real function $f$, we indicate with $f_{c}$ its complexif\-ication, which is the complex function whose real part coincides with $f$ and whose imaginary part vanishes. It is then not dif\-f\-icult to check that the following relations hold:
\begin{align} \label{Lorena rel}
(f \cdot g)_{c} &= f_{c} \cdot g_{c} \quad , & (f \ast g)_{c} &= f_{c} \ast g_{c} \quad .
\end{align}
Let us note furthermore that, from the expressions \eqref{complex Fourier transform} and \eqref{real Fourier transform}, the real Fourier transform and the complex Fourier transform are linked by the relation
\begin{equation} \label{Violetta direct rel}
\mscr{F}_{_{\!\mbbR}}[f] = \Re\, \mscr{F}\big[ f_{c} \big] + \Im\, \mscr{F}\big[ f_{c} \big] \quad ,
\end{equation}
and inversely
\begin{subequations} \label{Violetta inverse rel}
\begin{align}
\Re\, \mscr{F}\big[ f_{c} \big] &= \frac{1}{2}\, \Big( \mscr{F}_{_{\!\mbbR}}[f] + \big(P \circ \mscr{F}_{_{\!\mbbR}}\big)[f] \Big) \quad , \\[2mm]
\Im\, \mscr{F}\big[ f_{c} \big] &= \frac{1}{2}\, \Big( \mscr{F}_{_{\!\mbbR}}[f] - \big(P \circ \mscr{F}_{_{\!\mbbR}}\big)[f] \Big)\,  \quad .
\end{align}
\end{subequations}
The symbol $P$ denotes the parity operator, which is def\-ined as follows
\begin{align*}
P &: \mscr{S} \to \mscr{S} \quad , & P[f](\vectx) &= f(-\vectx) \quad .
\end{align*}
We adopt a condensed notation according to which $P\! \mscr{F}_{_{\!\mbbR}}$ denotes the composition $P \circ \mscr{F}_{_{\!\mbbR}}\,$.

\smallskip
The proposition for the real Fourier transform, which is analogous to the proposition \ref{prop: complex Fourier transform of product and convolution}, reads:
\begin{proposition}[real Fourier transform and convolution] \label{prop: real Fourier transform of product and convolution}
Let $f , g \in \mscr{S}$. Then
\begin{multline} \label{real Fourier transform of convolution}
\frac{1}{\sqrt{(2 \pi)^{n}}} \,\, \mscrF_{_{\!\mbbR}} \big[ f \ast g \big] = \frac{1}{2}\, \bigg( \mscrF_{_{\!\mbbR}}[f] \cdot \mscrF_{_{\!\mbbR}}[g] + \mscrF_{_{\!\mbbR}}[f] \cdot P\! \mscrF_{_{\!\mbbR}}[g] \,\, + \\[2mm]
+ P\! \mscrF_{_{\!\mbbR}}[f] \cdot \mscrF_{_{\!\mbbR}}[g] - P\! \mscrF_{_{\!\mbbR}}[f] \cdot P\! \mscrF_{_{\!\mbbR}}[g] \bigg) \quad ,
\end{multline}
and
\begin{multline} \label{real Fourier transform of product}
\mscrF_{_{\!\mbbR}} \big[ f \cdot g \big] = \frac{1}{\sqrt{(2 \pi)^{n}}} \,\, \frac{1}{2}\, \bigg( \mscrF_{_{\!\mbbR}}[f] \ast \mscrF_{_{\!\mbbR}}[g] + \mscrF_{_{\!\mbbR}}[f] \ast P\! \mscrF_{_{\!\mbbR}}[g] \,\, + \\[2mm]
+ P\! \mscrF_{_{\!\mbbR}}[f] \ast \mscrF_{_{\!\mbbR}}[g] - P\! \mscrF_{_{\!\mbbR}}[f] \ast P\! \mscrF_{_{\!\mbbR}}[g] \bigg) \quad .
\end{multline}
\end{proposition}
\begin{proof}
The idea of the proof is to express the real Fourier transform in terms of the complex one, by using the relations \eqref{Lorena rel}--\eqref{Violetta inverse rel}, and resort to the proposition \ref{prop: complex Fourier transform of product and convolution}. The proof is elementary and straightforward but, being a bit cumbersome, we conf\-ine it to the appendix \ref{app: proof}.
\end{proof}

\subsection{Comparison and spherical symmetry}

It is worthwhile to comment on the dif\mbox{}ference between the relations \eqref{complex Fourier transform of convolution}--\eqref{complex Fourier transform of product} and \eqref{real Fourier transform of convolution}--\eqref{real Fourier transform of product}, which correspond respectively to the complex and real def\mbox{}initions of Fourier transform. On the face of it, the relations \eqref{complex Fourier transform of convolution}--\eqref{complex Fourier transform of product} are much simpler than the relations \eqref{real Fourier transform of convolution}--\eqref{real Fourier transform of product}, and this may be regarded as a motivation to use the complex def\-inition of Fourier transform even when dealing with real functions.

Such an argument would be, however, misleading. A fair way to compare the two situations is to express the complex Fourier transform in terms of its real components, thereby comparing objects of the same ``degree of complexity''. This leads to using the Fourier transform $\mcal{F}$ as a mean of comparison to $\mscr{F}_{_{\!\mbbR}}\,$. Expressing the relations \eqref{complex Fourier transform of convolution}--\eqref{complex Fourier transform of product} in terms of the real and imaginary parts of the complex Fourier transform, we get
\begin{subequations} \label{naive real Fourier transform of convolution}
\begin{align}
\frac{1}{\sqrt{(2 \pi)^{n}}} \,\, \mcal{F}_{_{\!1\!}}\big[ f \ast g \big] &= \mcal{F}_{_{\!1\!}}\big[ f \big] \cdot \mcal{F}_{_{\!1\!}}\big[ g \big] - \mcal{F}_{_{2\!}}\big[ f \big] \cdot \mcal{F}_{_{2\!}}\big[ g \big] \quad , \\[2mm]
\frac{1}{\sqrt{(2 \pi)^{n}}} \,\, \mcal{F}_{_{2\!}}\big[ f \ast g \big] &= \mcal{F}_{_{\!1\!}}\big[ f \big] \cdot \mcal{F}_{_{2\!}}\big[ g \big] + \mcal{F}_{_{2\!}}\big[ f \big] \cdot \mcal{F}_{_{\!1\!}}\big[ g \big] \quad ,
\end{align}
\end{subequations}
and
\begin{subequations} \label{naive real Fourier transform of product}
\begin{align}
\mcal{F}_{_{\!1\!}}\big[ f \cdot g \big] &= \frac{1}{\sqrt{(2 \pi)^{n}}} \,\, \Big( \mcal{F}_{_{\!1\!}}\big[ f \big] \ast \mcal{F}_{_{\!1\!}}\big[ g \big] - \mcal{F}_{_{2\!}}\big[ f \big] \ast \mcal{F}_{_{2\!}}\big[ g \big] \Big) \quad , \\[2mm]
\mcal{F}_{_{2\!}}\big[ f \cdot g \big] &= \frac{1}{\sqrt{(2 \pi)^{n}}} \,\, \Big( \mcal{F}_{_{\!1\!}}\big[ f \big] \ast \mcal{F}_{_{2\!}}\big[ g \big] + \mcal{F}_{_{2\!}}\big[ f \big] \ast \mcal{F}_{_{\!1\!}}\big[ g \big] \Big) \quad .
\end{align}
\end{subequations}
It is highly questionable that the relations \eqref{naive real Fourier transform of convolution}--\eqref{naive real Fourier transform of product} are simpler than the relations \eqref{real Fourier transform of convolution}--\eqref{real Fourier transform of product}. The former have only two addends on the right hand side, but have twice the number of equations. Moreover, the right hand sides of \eqref{naive real Fourier transform of convolution}--\eqref{naive real Fourier transform of product} display four independent quantities ($\mcal{F}_{_{\!1}}[f]\,$, $\mcal{F}_{_{2}}[g]\,$, $\mcal{F}_{_{\!1}}[f]$ and $\mcal{F}_{_{2}}[g]$) while the right hand sides of \eqref{real Fourier transform of convolution}--\eqref{real Fourier transform of product} in some sense display only two, since $P\! \mscrF_{_{\!\mbbR}}[f]$ and $P\! \mscrF_{_{\!\mbbR}}[g]$ can be obtained from $\mscrF_{_{\!\mbbR}}[f]$ and $\mscrF_{_{\!\mbbR}}[g]$ by a parity transformation (although, linearly speaking, they are independent).

\smallskip
Furthermore, a signif\-icant simplif\-ication in the relations \eqref{real Fourier transform of convolution}--\eqref{real Fourier transform of product} takes place when (at least) one of the two real functions is parity-even. To prove this we need a little lemma:
\begin{lemma}
The real Fourier transformation $\mscr{F}_{_{\!\mbbR}}$ and the parity operator $P$, as operators on $\mscr{S}$, commute.
\end{lemma}
\begin{proof}
Let $f \in \mscr{S}$. By explicit evaluation we have
\begin{multline*}
\big( \mscr{F}_{_{\!\mbbR}} \circ P \big)[f](\vecty) = \frac{1}{\sqrt{(2 \pi)^{n}}} \int_{\mbbR^{n}} \! f(-\vectx) \, \Big( \cos (\vecty \cdot \vectx) - \sin (\vecty \cdot \vectx) \Big) \, d\vectx = \\
= \frac{1}{\sqrt{(2 \pi)^{n}}} \int_{\mbbR^{n}} \! f(\vectz) \, \Big( \cos(- \vecty \cdot \vectz) - \sin( - \vecty \cdot \vectz) \Big) \, d\vectz = \mscr{F}_{_{\!\mbbR}}[f](-\vecty) \quad ,
\end{multline*}
where we changed integration variables $\vectx \to \vectz = - \vectx$ in passing from the f\-irst to the second line.
\end{proof}
We then arrive at the pleasing result:
\begin{corollary} \label{cor: real Fourier transform of product and convolution}
Let $f , g \in \mscr{S}$, and at least one of the two functions be parity-even. Then
\begin{equation} \label{real Fourier transform of convolution special}
\frac{1}{\sqrt{(2 \pi)^{n}}} \,\, \mscrF_{_{\!\mbbR}}\big[ f \ast g \big] = \mscrF_{_{\!\mbbR}}[f] \cdot \mscrF_{_{\!\mbbR}}[g] \quad ,
\end{equation}
and
\begin{equation} \label{real Fourier transform of product special}
\mscrF_{_{\!\mbbR}}\big[ f \cdot g \big] = \frac{1}{\sqrt{(2 \pi)^{n}}} \,\, \mscrF_{_{\!\mbbR}}[f] \ast \mscrF_{_{\!\mbbR}}[g] \quad .
\end{equation}
\end{corollary}
\begin{proof}
The thesis follows trivially from the relations \eqref{real Fourier transform of convolution} and \eqref{real Fourier transform of product}, taking into account that $\mscrF_{_{\!\mbbR}}$ and $P$ commute.
\end{proof}

For example, if one of the functions is spherically symmetric (think of the convolution of a function by a spherically symmetric window function), the relations \eqref{real Fourier transform of convolution special}--\eqref{real Fourier transform of product special} hold.

\section{Conclusions}
\label{sec: conclusions}

In this article, we proposed a novel def\-inition of Fourier transform of a real function, with the aim of avoiding the downside of the usual def\-inition of producing a transformed function with two real components.

We achieved this aim by introducing the notion of ``real Fourier transform'', and we showed that it enjoys the good properties of the usual complex Fourier transform. In particular the inversion theorem holds, the transformation preserves the parity of a function, and preserves the $L^{2}$ inner product. We discussed the extension of the transformation to a map $L^{2} \to L^{2}$ and to a map $L^{1} \to C_{_{0\!}}\,$, and also the interplay between the real Fourier transform and the operation of convolution.

Some surprises were encountered: the distinction between transform and anti-transform is no longer necessary, because the novel transform is an involution. Also, a pleasing simplif\-ication in the formulas for the real Fourier transform of convolutions and products arises, whenever (at least) one of the two functions is parity-even.

\bigskip

\section*{Acknowledgments}

The author acknowledges partial f\mbox{}inancial support, during an early stage of this work, from the Funda\c{c}\~{a}o de Amparo \`{a} Pesquisa do Estado do Rio de Janeiro (FAPERJ, Brazil) under the Programa de Apoio \`{a} Doc\^{e}ncia (PAPD) program.

\appendix

\section{Proof of the proposition \ref{prop: real Fourier transform of product and convolution}}
\label{app: proof}

We describe in some detail the proof of the proposition \ref{prop: real Fourier transform of product and convolution}. We need to prove that, for every $f , g \in \mscr{S}$, we have
\begin{multline} \label{app real Fourier transform of convolution}
\frac{1}{\sqrt{(2 \pi)^{n}}} \,\, \mscrF_{_{\!\mbbR}} \big[ f \ast g \big] = \frac{1}{2}\, \bigg( \mscrF_{_{\!\mbbR}}[f] \cdot \mscrF_{_{\!\mbbR}}[g] + \mscrF_{_{\!\mbbR}}[f] \cdot P\! \mscrF_{_{\!\mbbR}}[g] \,\, + \\[2mm]
+ P\! \mscrF_{_{\!\mbbR}}[f] \cdot \mscrF_{_{\!\mbbR}}[g] - P\! \mscrF_{_{\!\mbbR}}[f] \cdot P\! \mscrF_{_{\!\mbbR}}[g] \bigg) \quad ,
\end{multline}
and
\begin{multline} \label{app real Fourier transform of product}
\mscrF_{_{\!\mbbR}} \big[ f \cdot g \big] = \frac{1}{\sqrt{(2 \pi)^{n}}} \,\, \frac{1}{2}\, \bigg( \mscrF_{_{\!\mbbR}}[f] \ast \mscrF_{_{\!\mbbR}}[g] + \mscrF_{_{\!\mbbR}}[f] \ast P\! \mscrF_{_{\!\mbbR}}[g] \,\, + \\[2mm]
+ P\! \mscrF_{_{\!\mbbR}}[f] \ast \mscrF_{_{\!\mbbR}}[g] - P\! \mscrF_{_{\!\mbbR}}[f] \ast P\! \mscrF_{_{\!\mbbR}}[g] \bigg) \quad .
\end{multline}
\begin{proof}
As we mentioned in the main text, the idea of the proof is to resort to the proposition \ref{prop: complex Fourier transform of product and convolution}, by using the relations \eqref{Lorena rel}--\eqref{Violetta inverse rel}. Let us consider f\-irst the relation \eqref{app real Fourier transform of convolution}. Using the relation \eqref{Violetta direct rel}, the second of the relations \eqref{Lorena rel} and the relation \eqref{complex Fourier transform of convolution}, we can express the quantity $\mscrF_{_{\!\mbbR}}[f \ast g]$ as follows:
\begin{multline} \label{Camila}
\mscrF_{_{\!\mbbR}} \big[ f \ast g \big] = \Re\, \mscr{F}\big[ (f \ast g)_{c} \big] + \Im\, \mscr{F} \big[ (f \ast g)_{c} \big] = \Re\, \mscr{F}\big[ f_{c} \ast g_{c} \big] + \Im\, \mscr{F} \big[ f_{c} \ast g_{c} \big] = \\[2mm]
= \, \sqrt{(2 \pi)^{n}} \,\, \bigg[ \, \Re\, \Big( \mscrF \big[ f_{c} \big] \cdot \mscrF \big[ g_{c} \big] \Big) + \Im\, \Big( \mscrF \big[ f_{c} \big] \cdot \mscrF \big[ g_{c} \big] \Big) \bigg] \quad .
\end{multline}
Noting that
\begin{align*}
\Re\, \Big( \mscrF \big[ f_{c} \big] \cdot \mscrF \big[ g_{c} \big] \Big) &= \Re\, \mscrF \big[ f_{c} \big] \cdot \Re\, \mscrF \big[ g_{c} \big] - \Im\, \mscrF \big[ f_{c} \big] \cdot \Im\, \mscrF \big[ g_{c} \big] \quad , \\[2mm]
\Im\, \Big( \mscrF \big[ f_{c} \big] \cdot \mscrF \big[ g_{c} \big] \Big) &= \Re\, \mscrF \big[ f_{c} \big] \cdot \Im\, \mscrF \big[ g_{c} \big] + \Im\, \mscrF \big[ f_{c} \big] \cdot \Re\, \mscrF \big[ g_{c} \big] \quad ,
\end{align*}
and expressing $\Re\, \mscrF \big[ f_{c} \big]\,$, $\Im\, \mscrF \big[ f_{c} \big]\,$, $\Re\, \mscrF \big[ g_{c} \big]$ and $\Im\, \mscrF \big[ g_{c} \big]$ via the relations \eqref{Violetta inverse rel}, the relation \eqref{Camila} leads to an expression for $\mscrF_{_{\!\mbbR}} \big[ f \ast g \big]$ in terms of $\mscrF_{_{\!\mbbR}}[f]\,$, $P\!\mscrF_{_{\!\mbbR}}[f]\,$, $\mscrF_{_{\!\mbbR}}[g]$ and $P\!\mscrF_{_{\!\mbbR}}[g]\,$. After a fair number of straightforward simplif\-ications, the relation \eqref{app real Fourier transform of convolution} is obtained.

The proof of the relation \eqref{app real Fourier transform of product} is similar in spirit. Using the relation \eqref{Violetta direct rel}, the f\-irst of the relations \eqref{Lorena rel} and the relation \eqref{complex Fourier transform of product}, we can express the quantity $\mscrF_{_{\!\mbbR}}[f \cdot g]$ as follows
\begin{multline} \label{Camila 2}
\mscrF_{_{\!\mbbR}} \big[ f \cdot g \big] = \Re\, \mscr{F}\big[ (f \cdot g)_{c} \big] + \Im\, \mscr{F} \big[ (f \cdot g)_{c} \big] = \Re\, \mscr{F}\big[ f_{c} \cdot g_{c} \big] + \Im\, \mscr{F} \big[ f_{c} \cdot g_{c} \big] = \\[2mm]
= \frac{1}{\sqrt{(2 \pi)^{n}}} \,\, \Re\, \Big( \mscrF \big[ f_{c} \big] \ast \mscrF \big[ g_{c} \big] \Big) + \frac{1}{\sqrt{(2 \pi)^{n}}} \,\, \Im\, \Big( \mscrF \big[ f_{c} \big] \ast \mscrF \big[ g_{c} \big] \Big) \quad .
\end{multline}
Noting that
\begin{align*}
\Re\, \Big( \mscrF \big[ f_{c} \big] \ast \mscrF \big[ g_{c} \big] \Big) &= \Re\, \mscrF \big[ f_{c} \big] \ast \Re\, \mscrF \big[ g_{c} \big] - \Im\, \mscrF \big[ f_{c} \big] \ast \Im\, \mscrF \big[ g_{c} \big] \quad , \\[2mm]
\Im\, \Big( \mscrF \big[ f_{c} \big] \ast \mscrF \big[ g_{c} \big] \Big) &= \Re\, \mscrF \big[ f_{c} \big] \ast \Im\, \mscrF \big[ g_{c} \big] + \Im\, \mscrF \big[ f_{c} \big] \ast \Re\, \mscrF \big[ g_{c} \big] \quad ,
\end{align*}
and again expressing $\Re\, \mscrF \big[ f_{c} \big]\,$, $\Im\, \mscrF \big[ f_{c} \big]\,$, $\Re\, \mscrF \big[ g_{c} \big]$ and $\Im\, \mscrF \big[ g_{c} \big]$ via the relations \eqref{Violetta inverse rel}, the relation \eqref{Camila 2} leads to an expression for $\mscrF_{_{\!\mbbR}} \big[ f \cdot g \big]$ in terms of $\mscrF_{_{\!\mbbR}}[f]\,$, $P\!\mscrF_{_{\!\mbbR}}[f]\,$, $\mscrF_{_{\!\mbbR}}[g]$ and $P\!\mscrF_{_{\!\mbbR}}[g]\,$. Also in this case a fair number of straightforward simplif\-ications lead to the relation \eqref{app real Fourier transform of product}.
\end{proof}

\section{Modes and Fourier expansions}
\label{app: modes}
 
Let us consider a real function $f$ for which the Fourier anti-transform can be expressed by an integral representation. The def\-initions \eqref{naive real Fourier antitransform} and \eqref{real Fourier transf and anti-transf}, taking into account the proposition \ref{prop: involutivity}, can be understood as providing expansions of $f$ over continuous sets of modes
\begin{equation}  \label{naive real Fourier expansion}
f(\vectx) = \frac{1}{\sqrt{(2 \pi)^{n}}} \int_{\mbbR^{n}} \! \Big( \ti{f}_{_{1\!}}(\vecty) \, \cos_{\vecty}(\vectx) - \ti{f}_{_{2\!}}(\vecty) \, \sin_{\vecty}(\vectx) \Big) \, d\vecty \quad ,
\end{equation}
and
\begin{equation} \label{real Fourier expansion}
f(\vectx) = \frac{1}{\sqrt{(2 \pi)^{n}}} \int_{\mbbR^{n}} \! \ti{f}(\vecty) \, \s_{\vecty}(\vectx) \, d\vecty \quad ,
\end{equation}
where we introduced the modes
\begin{align*}
\cos_{\vecty}(\vectx) &= \cos (\vecty \cdot \vectx) \quad , & \sin_{\vecty}(\vectx) &= \sin (\vecty \cdot \vectx) \quad , \\[2mm]
\s_{\vecty}(\vectx) &= \cos (\vecty \cdot \vectx) - \sin (\vecty \cdot \vectx) \quad .
\end{align*}
From this point of view, $\cos_{\vecty}\,$, $\sin_{\vecty}$ and $\s_{\vecty}$ are to be understood as indexed families of functions $\mbbR^{n} \to \mbbR\,$, where $\vecty \in \mbbR^{n}$ is the index, while $\ti{f}_{_{1\!}}\,$, $\ti{f}_{_{2\!}}$ and $\ti{f}$ play the role of ``amplitude'' functions. The terminology ``continuous set of modes'' is somewhat improper, because the modes do not belong to the same functional space as the function itself, nevertheless it is standard. 

It is worthwhile to comment brief\-ly about the dif\-ferences between the two expansions, and especially about the fact that $f$ does not determine the couple $(\ti{f}_{_{1\!}}\, , \ti{f}_{_{2}})$ uniquely, while it does determine $\ti{f}$ uniquely. Of course, ultimately this depends on the conditions required to invert the Fourier anti-transform, but here we want to look at this fact from the point of view of the properties of the families of modes.

\smallskip
To facilitate the discussion, we refer to $\cos_{\vecty}$ and $\sin_{\vecty}$ as the ``sinusoidal modes'', and to $\s_{\vecty}$ as the ``sigma'' modes. So, \eqref{naive real Fourier expansion} can be seen as an expansion over the sinusoidal modes, while \eqref{real Fourier expansion} can be seen as an expansion over the sigma modes. It is also useful to introduce an equivalence relation $\sim$ between vectors in $\mbbR^{n}$ such that 
\begin{equation} \label{equivalence relation}
\vecty \sim \vectq \quad \text{if\mbox{}f} \quad \big[ \, \vecty = \vectq \,\,\, \textup{or} \,\,\, \vecty = - \vectq \, \big] \quad .
\end{equation}
It is easy to see that the modes $\s_{\vecty}\,$, $\sin_{\vectk}$ and $\cos_{\vectq}\,$, whose indexes belong to dif\-ferent equivalence classes of $\sim\,$, are linearly independent. On the other hand, assuming $\vecty \neq \vect{0}\,$, the equivalence class $[\vecty]$ contains two sigma modes, and four sinusoidal modes. The latter are not linearly independent, since $\cos_{\vecty}$ and $\cos_{-\vecty}$ are proportional one to the other, and the same is true of $\sin_{\vecty}$ and $\sin_{-\vecty}\,$. This implies that there exists a non-trivial transformation on the couple $(\ti{f}_{_{1\!}}\, , \ti{f}_{_{2}})$ which leaves the integral \eqref{naive real Fourier expansion} unchanged. In fact, if we substitute the couple $(\ti{f}_{_{1\!}}\, , \ti{f}_{_{2}})$ with the couple $(\ti{g}_{_{1\!}}\, , \ti{g}_{_{2}})\,$, the integral remains unchanged if
\begin{align*}
\Big( \ti{f}_{_{1\!}}(\vecty) - \ti{g}_{_{1\!}}(\vecty) \Big) \, \cos_{\vecty}(\vectx) + \Big( \ti{f}_{_{1\!}}(-\vecty) - \ti{g}_{_{1\!}}(-\vecty) \Big) \, \cos_{-\vecty}(\vectx) &= 0 \quad , \\[2mm]
\Big( \ti{f}_{_{2\!}}(\vecty) - \ti{g}_{_{2\!}}(\vecty) \Big) \, \sin_{\vecty}(\vectx) + \Big( \ti{f}_{_{2\!}}(-\vecty) - \ti{g}_{_{2\!}}(-\vecty) \Big) \, \sin_{-\vecty}(\vectx) &= 0 \quad ,
\end{align*}
and these conditions can always be satisf\-ied non-trivially because of the linear dependence. No such freedom exists for the expansion \eqref{real Fourier expansion}, because the modes $\s_{\vecty}$ and $\s_{-\vecty}$ are linearly independent.

From the point of view of the expansion \eqref{naive real Fourier expansion}, there are two straightforward approaches to remedy this problem. One is to continue using sinusoidal modes, and integrate only over half of the reciprocal space. That is, to leave the expression \eqref{naive real Fourier expansion} unchanged but for the fact that the domain of integration becomes the quotient $\mbbR^{n}/\!\!\sim$ instead of $\mbbR^{n}$. The other, equivalent, approach is integrate over the whole space $\mbbR^{n}$, but at the same time restrict the function $\ti{f}_{_{1\!}}$ to be parity-even and the function $\ti{f}_{_{2\!}}$ to be parity-odd. This restriction ef\-fectively identif\-ies the mode $\cos_{-\vecty}$ with $\cos_{\vecty}\,$, and ef\-fectively identif\-ies the mode $\sin_{-\vecty}$ with $- \sin_{\vecty}\,$.

\smallskip
One may however prefer to avoid restricting the domain of integration to the quotient $\mbbR^{n}/\!\!\sim\,$, and also avoid imposing that the amplitude functions $\ti{f}_{_{1\!}}$ and $\ti{f}_{_{2\!}}$ have def\-inite parity. In this case one is forced to act on the modes, extracting somehow two linearly independent modes from the four sinusoidal ones (at each $[\vecty]$). From this perspective, the introduction of the real Fourier transform can be motivated by the desire of f\-inding a real expansion over ``modif\-ied sinusoidal'' modes, such that the domain of integration in the expansion is the whole space $\mbbR^{n}$, and yet the modes are linearly independent. To achieve this aim, we should look for a set of modes $( \s_{\vecty})_{\vecty \in \mbbR^{n}}\,$, such that $\s_{\vecty}$ and $\s_{-\vecty}$ are linearly independent and generate the four modes $\cos_{\vecty}\,$, $\cos_{-\vecty}\,$, $\sin_{\vecty}$ and $\sin_{-\vecty}\,$. The parity properties of the sinusoidal modes imply that it is suf\-f\-icient to generate the two modes $\cos_{\vecty}$ and $\sin_{\vecty}\,$.

To f\-ind such sigma modes, it is useful to ref\-lect on the fact that the problem with the sinusoidal modes is that the index transformation $\vecty \to - \vecty$ produces functions proportional to the initial ones
\begin{align*}
\cos_{\vecty} &\to \cos_{-\vecty} = \cos_{\vecty} \quad , & \sin_{\vecty} &\to \sin_{-\vecty} = - \sin_{\vecty} \quad .
\end{align*}
A promising way out is then to def\-ine the new modes by summing functions of opposite parity, so that the transformation $\vecty \to - \vecty$ has a non-trivial ef\-fect. For example, choosing $\s_{\vecty} = \cos_{\vecty} - \sin_{\vecty}\,$, under the transformation $\vecty \to - \vecty$ one has
\begin{equation*}
\s_{-\vecty} = \cos_{-\vecty} - \sin_{-\vecty} = \cos_{\vecty} + \sin_{\vecty} \quad ,
\end{equation*}
which is \emph{not} proportional to $\s_{\vecty}\,$. This solves the problem, because the modes $\s_{\vecty}$ and $\s_{-\vecty}$ manifestly generate $\cos_{\vecty}$ and $\sin_{\vecty}\,$, as desired.

\smallskip
It worthwhile to mention that, besides being linearly independent, the family of modes
\begin{equation*}
\Bigg( \frac{1}{\sqrt{(2 \pi)^{n}}} \,\, \s_{\vecty} \Bigg)_{\vecty \in \mbbR^{n}}
\end{equation*}
is orthonormal in Dirac's sense. In other words, these modes are ``normalized to Dirac deltas''.

\end{document}